\documentclass{birkjour}
\usepackage{graphicx}
\usepackage{amsfonts, amssymb, amsmath, mathrsfs, latexsym}
\usepackage[all]{xy}
\usepackage[dvipsnames]{xcolor}
\usepackage{mathtools}
\usepackage{tikz}
\usetikzlibrary{positioning}
\usepackage{xcolor}
\usepackage[normalem]{ulem}

\ifpdf
	\usepackage[unicode,pdftex]{hyperref}
	\input glyphtounicode
\else
	\usepackage[unicode,dvipdfm]{hyperref}
\fi

\newtheorem{thm}{Theorem}[section]
\newtheorem{prop}[thm]{Proposition}
\newtheorem{lem}[thm]{Lemma}

\theoremstyle{definition}

\newtheorem{rmk}[thm]{Remark}
\newtheorem{exam}[thm]{Example}
\newtheorem*{mcs}{Modified Assertion of Severi}

\newcommand{\PP}{\ensuremath{\mathbb{P}}}

\begin{document}

\title[On  the Hilbert scheme of  linearly normal curves]
{On  the Hilbert scheme of  linearly normal curves in $\mathbb{P}^r$ of relatively high degree}

\thanks{The first and second named authors were supported in part by GNSAGA of INdAM and by PRIN 2017 ``Moduli Theory and Birational Classification".  The third named author was supported in part by National Research Foundation of South Korea 
(2019R1I1A1A01058457).}

\author[E. Ballico]{Edoardo Ballico}
\address{Dipartimento di Matematica, Universit\`a degli Studi di Trento\\
Via Sommarive 14, 38123 Povo, Italy}
\email{edoardo.ballico@unitn.it}

\author[C. Fontanari]{Claudio Fontanari}
\address{Dipartimento di Matematica, Universit\`a degli Studi di Trento\\
Via Sommarive 14, 38123 Povo, Italy}
\email{claudio.fontanari@unitn.it}

\author[C. Keem]{Changho Keem}
\address{
Department of Mathematics,
Seoul National University\\
Seoul 151-742,  
South Korea}
\email{ckeem1@gmail.com}



\date{\today}
\maketitle

\begin{abstract}
Let $\mathcal{H}_{d,g,r}$ be the Hilbert scheme parametrizing smooth irreducible and non-degenerate curves of degree $d$ and genus $g$ in $\PP^r$. 
We denote by $\mathcal{H}^\mathcal{L}_{d,g,r}$ the union of those components of $\mathcal{H}_{d,g,r}$ whose general element is linearly normal
and we show that any non-empty $\mathcal{H}^\mathcal{L}_{d,g,r}$ ($d\ge g+r-3$) is irreducible for an extensive range of triples $(d,g,r)$
beyond the Brill-Noether range.
This establishes the validity of a suitably modified assertion of Severi regarding the irreducibility of the Hilbert scheme $\mathcal{H}^\mathcal{L}_{d,g,r}$ 
of linearly normal curves for $g+r-3\le d\le g+r$,  $r\ge 3$, and $g \ge 2r+3$ if $d=g+r-3$.
\end{abstract}
\section{\quad An overview, preliminaries and basic set-up}

Given non-negative integers $d$, $g$ and $r\ge 3$, let $\mathcal{H}_{d,g,r}$ denote the Hilbert scheme of smooth curves parametrizing smooth irreducible and non-degenera\-te curves of degree $d$ and genus $g$ in $\PP^r$.

\vskip 4pt
The problem of the irreducibility of $\mathcal{H}_{d,g,r}$ was first addressed by Severi in the paper \cite{Sev1}, hastily written in Italian in 1915 after the outbreak of the First World War, and then resumed in more detail in the book \cite{Sev}, published in German in 1921. Severi claims with an incomplete proof that $\mathcal{H}_{d,g,r}$ is irreducible for $d\ge g+r$ (in 
\cite{Sev1}, \S 3: \emph{Per $n \ge p+r$, le curve $C^n_p$ di $S_r$ ($r \ge 2$) formano una sola famiglia}; in \cite{Sev}, Anhang G, p. 369: \emph{Im Raum $S_r$ ($r > 2$) bilden die irreduziblen Kurven $C$ von der $n$-ten Ordnung und der Geschlecht $p$, wenn  $n \ge p+r$ ist, eine einzige irreduzible Familie $V$}); more generally, Severi asserts that $\mathcal{H}_{d,g,r}$ is irreducible in the Brill-Noether range $\rho (d,g,r):=g-(r+1)(g-d+r)\ge 0$  (in \cite{Sev1}, \S 3: \emph{Per $p > n - r \ge \frac{r}{r+1} p$ le $C^n_p$ di $S_r$ 
formano una sola famiglia}; in \cite{Sev}, Anhang G, p. 399: \emph{Wenn $r + p > n \ge r + \frac{r}{r+1} p$ ist, so bilden die irreduziblen Kurven $n$-ter Ordnung vom Geschlecht $p$
des Raumes $S_r$ eine Familie $W$}).

\vskip 4pt
After Severi, the irreducibility of $\mathcal{H}_{d,g,r}$ has been studied by several authors. Ein proved Severi's first claim for $r=3$ and $r=4$; cf.  \cite[Theorem 4]{E1} 
and \cite[Theorem 7]{E2}. On the other hand, an example of Harris showed that the claim is incorrect for $r \ge 6$; cf. \cite[Proposition 9]{E2}.

\vskip 4pt
For families of curves in $\mathbb{P}^3$ of lower degree $d\le g+2$, the most updated result is that any non-empty $\mathcal{H}_{d,g,3}$ is irreducible for every $d\ge g$; cf. \cite[Theorem 1.5]{KK},  \cite[Proposition 2.1 and Proposition 3.2 ]{KKL}, \cite[Theorem 3.1]{I} and \cite{KKy1}.

\vskip 4pt
For families of curves in $\PP^4$ of lower degree $d\le g+3$, Iliev proved the irreducibility of $\mathcal{H}_{d,g,4}$ for $d=g+3$, $g\ge 5$ and $d=g+2$, $g\ge 11$; cf. \cite{I}. 
Quite recently, there has been a minor extension of the result of Iliev regarding the irreducibility of $\mathcal{H}_{g+2,g,4}$ for low genus cases: namely, $\mathcal{H}_{g+2,g,4}$ is irreducible and generically reduced for any genus $g$ as long as $\mathcal{H}_{g+2,g,4}$ is non-empty; cf. \cite[Corollary 2.2]{KKy2}.

\vskip 4pt
Despite these partial positive results, Severi's second assertion turns out to be false in general: counterexamples have been provided by Eisenbud and Harris (see \cite{EH}), 
by Mezzetti and Sacchiero (see \cite{MS}) and by the third named author (see \cite{Keem}). On the other hand, there has been an attempt to look at the irreducibility problem 
for the Hilbert scheme of smooth curves from a different perspective. As a matter of fact, in all counterexamples quoted above the irreducible components unexpected by 
Severi parametrize non linearly normal curves and indeed there are several sources in the literature suggesting that what Severi indeed had in mind in his original assertion 
might have been the following statement; cf. \cite[page 489]{CS} or the review by Lopez of \cite{Keem} on MathSciNet; AMS Mathematical Reviews MR1221726(95a:14026).

\begin{mcs}  A nonempty $\mathcal{H}^\mathcal{L}_{d,g,r}$ is irreducible for any triple $(d,g,r)$ in the Brill-Noether range  $\rho (d,g,r)=g-(r+1)(g-d+r)\ge 0$, where $\mathcal{H}^\mathcal{L}_{d,g,r}$ is the union of those components of ${\mathcal{H}}_{d,g,r}$ whose general element is linearly normal.
\end{mcs}

Reflecting such a slightly different point of view, the authors of \cite{KK3} showed that every non-empty $\mathcal{H}_{g+1,g,4}^\mathcal{L}$ is irreducible unless $g=9$; cf.
\cite[Theorem 2.1, Theorem 2.2, Remark 1.1 and Remark 2.5]{KK3} for several sporadic reducible examples of $\mathcal{H}^\mathcal{L}_{d,g,3}$ and $\mathcal{H}^\mathcal{L}_{d,g,4}$.

Maintaining the same spirit and following the same line of ideas adopted in \cite{KK3}, the purpose of this article is to extend the result obtained in \cite{KK3} to the case $r\ge 5$. 

\vskip 4pt
The organization of this paper is as follows. After we briefly mention and recall several basic preliminaries in the remainder of this section, we start the next section by proving the irreducibility of  $\mathcal{H}^\mathcal{L}_{g+r,g,r}$ and $\mathcal{H}^\mathcal{L}_{g+r-1,g,r}$ which are relatively easy to verify.  We  then proceed to show the irreducibility of $\mathcal{H}^\mathcal{L}_{g+r-2,g,r}$ for every possible triple $(g+r-2,g,r)$ as long as the corresponding Hilbert scheme of linearly normal curves $\mathcal{H}^\mathcal{L}_{g+r-2,g,r}$ is non-empty; Theorem \ref{g+r-2}.
In the last part of the next section we prove the irreducibility of  $\mathcal{H}^\mathcal{L}_{g+r-3,g,r}$ under the mild restriction $g\ge 2r+3$; Theorem \ref{g+r-3}. In the course of the proof we characterize the residual series with respect to the canonical series of the complete linear series corresponding to the linearly normal curves under consideration. We use  the irreducibility of the Severi variety of plane curves in the final stage of the proof of Theorem \ref{g+r-3}; cf. \cite{AC2} and~\cite{H2}.

\vskip 4pt
For notation and conventions, we usually follow those in \cite{ACGH} and \cite{ACGH2}; e.g. $\pi (d,r)$ is the maximal possible arithmetic genus of an irreducible and non-degenerate curve of degree $d$ in $\PP^r$. Throughout we work over the field of complex numbers. 

\vskip 4pt
Before proceeding, we recall several related results which are rather well-known; cf. \cite{ACGH2}  or \cite[\S 1 and \S 2]{AC2}.
Let $\mathcal{M}_g$ be the moduli space of smooth curves of genus $g$. For any given isomorphism class $[C] \in \mathcal{M}_g$ corresponding to a smooth irreducible curve $C$, there exist a neighborhood $U\subset \mathcal{M}_g$ of the class $[C]$ and a smooth connected variety $\mathcal{M}$ which is a finite ramified covering $h:\mathcal{M} \to U$, as well as  varieties $\mathcal{C}$, $\mathcal{W}^r_d$ and $\mathcal{G}^r_d$ proper over $\mathcal{M}$ with the following properties:
\begin{enumerate}
\item[(1)] $\xi:\mathcal{C}\to\mathcal{M}$ is a universal curve, i.e. for every $p\in \mathcal{M}$, $\xi^{-1}(p)$ is a smooth curve of genus $g$ whose isomorphism class is $h(p)$,
\item[(2)] $\mathcal{W}^r_d$ parametrizes the pairs $(p,L)$ where $L$ is a line bundle of degree $d$ and $h^0(L) \ge r+1$ on $\xi^{-1}(p)$,
\item[(3)] $\mathcal{G}^r_d$ parametrizes the couples $(p, \mathcal{D})$, where $\mathcal{D}$ is possibly an incomplete linear series of degree $d$ and dimension $r$ on $\xi^{-1}(p)$ - which is usually denoted by $g^r_d$. 
\end{enumerate}

\vskip 4pt
Let $\widetilde{\mathcal{G}}$ ($\widetilde{\mathcal{G}}_\mathcal{L}$ resp.) be  the union of components of $\mathcal{G}^{r}_{d}$ whose general element $(p,\mathcal{D})$ of $\widetilde{\mathcal{G}}$ ($\widetilde{\mathcal{G}}_\mathcal{L}$ resp.) corresponds to a very ample (very ample and complete resp.) linear series $\mathcal{D}$ on the curve $C=\xi^{-1}(p)$. Note that an open subset of $\mathcal{H}_{d,g,r}$ consisting of points corresponding to smooth irreducible and non-degenerate curves is a $\mathbb{P}\textrm{GL}(r+1)$-bundle over an open subset of $\widetilde{\mathcal{G}}$. Hence the irreducibility of $\widetilde{\mathcal{G}}$ guarantees the irreducibility of $\mathcal{H}_{d,g,r}$. Likewise, the irreducibility of $\widetilde{\mathcal{G}}_\mathcal{L}$ ensures the irreducibility of 
 $\mathcal{H}_{d,g,r}^\mathcal{L}$.
We also make a note of the following well-known facts regarding the schemes $\mathcal{G}^{r}_{d}$ and $\mathcal{W}^r_d$; cf. \cite[Proposition 2.7, 2.8]{AC2}, \cite[2.a]{H1}, \cite[Ch. 21, \S 3, 5, 6, 11, 12]{ACGH2} and \cite[Theorem 1]{EH}. Following classical terminology, a base-point-fee linear series $g^r_d$ ($r\ge 2$) on a smooth curve $C$ is called birationally very ample when the morphism 
$C \rightarrow \mathbb{P}^r$ induced by  the $g^r_d$ is generically one-to-one (or birational) onto its image;  cf. \cite[p. 570]{EH2}.

\begin{prop}\label{facts}
For non-negative integers $d$, $g$ and $r$, let $\rho(d,g,r):=g-(r+1)(g-d+r)$ be the Brill-Noether number.
	\begin{enumerate}
	\item[\rm{(1)}] The dimension of any component of $\mathcal{G}^{r}_{d}$ is at least $3g-3+\rho(d,g,r)$ which is denoted by $\lambda(d,g,r)$. Moreover, if $\rho(d,g,r)\geq0$, there exists a unique component $\mathcal{G}_0$ of $\widetilde{\mathcal{G}}$ which dominates $\mathcal{M}$(or $\mathcal{M}_g$).
	\item[\rm{(2)}] Suppose $g>0$ and let $X$ be a component of $\mathcal{G}^{2}_{d}$ whose general element $(p,\mathcal{D})$ is such that $\mathcal{D}$ is a birationally very ample linear series on $\xi^{-1}(p)$. Then 
	\[\dim X=3g-3+\rho(d,g,2)=3d+g-9.\]
	\item[\rm{(3)}] $\mathcal{G}^{1}_{d}$ is smooth and irreducible of dimension $\lambda(d,g,1)$ if $g>1, d\ge 2$ and $d\le g+1$.
	\end{enumerate}
\end{prop}
\begin{rmk}\label{principal}
\begin{enumerate}
\item[(1)]
In the Brill-Noether range $\rho (d,g,r)\ge 0$, the unique component $\mathcal{G}_0$ of $\widetilde{\mathcal{G}}$ (and the corresponding component $\mathcal{H}_0$ of $\mathcal{H}_{d,g,r}$ as well) which dominates $\mathcal{M}$ or $\mathcal{M}_g$ is called the  ``principal component".  

\vskip 4pt
\noindent
\item[(2)] In the range $d\le g+r$ inside the Brill-Noether range $\rho (d,g,r)\ge 0$, the principal component $\mathcal{G}_0$ which has the expected dimension $\lambda (d,g,r)$ is one of the components of $\widetilde{\mathcal{G}}_\mathcal{L}$ (cf. \cite[2.1 page 70]{H1}),   and  hence $\widetilde{\mathcal{G}}_\mathcal{L}$ or 
 $\mathcal{H}_{d,g,r}^\mathcal{L}$ is non-empty in this range. Therefore what we are chasing after in the Brill-Noether range is if $\mathcal{G}_0$ is the only component of 
 $\widetilde{\mathcal{G}}_\mathcal{L}$.
\end{enumerate}
\end{rmk}

We recall that the family of plane curves of degree $d$ in $\mathbb{P}^2$ are naturally parametrised by the projective space $\mathbb{P}^N$, $N=\frac{d(d+3)}{2}$. Let $\Sigma_{d,g}\subset \PP^N$ be the Severi variety of plane curves of degree $d$ and geometric genus $g$. We also recall that a general point of $\Sigma_{d,g}$ corresponds to an  irreducible plane curve of degree $d$ having $\delta:=\frac{(d-1)(d-2)}{2}-g$ nodes and no other singularities. 
The following theorem of Harris is fundamental; cf. \cite[Theorem 10.7 and 10.12]{ACGH2} or \cite[Lemma 1.1, 1.3 and 2.3]{H2} .

\begin{thm}\label{severi}
$\Sigma_{d,g}$ is irreducible of dimension $3d+g-1=\lambda(d,g,2)+\dim\mathbb{P}\rm{GL}(3)$. 
\end{thm}

Denoting by $\mathcal{G'}\subset \mathcal{G}^{2}_{d}$ the 
union of components whose general element $(p,\mathcal{D})$  is such that $\mathcal{D}$ is birationally very ample  on $C=\xi^{-1}(p)$, we 
remark that an open subset of the Severi variety $\Sigma_{d,g}$ is a $\mathbb{P}\textrm{GL}(3)$-bundle over an open subset of  $\mathcal{G}'$.  Therefore, as  an immediate consequence of Theorem \ref{severi}, the irreducibility of  $\Sigma_{d,g}$ implies the 
irreducibility of the locus $\mathcal{G'}\subset \mathcal{G}^{2}_{d}$ and vice versa. We make a note of this  observation as the following lemma.
\begin{lem}\label{Gisirred}
Let $\mathcal{G}'\subset \mathcal{G}^{2}_{d}$ be the union of components whose general element $(p,\mathcal{D})$ is such that $\mathcal{D}$ is birationally very ample on $C=\xi^{-1}(p)$. Then $\mathcal{G}'$ is irreducible.
\end{lem}

We will utilize the  following upper bound of the dimension of an irreducible component of $\mathcal{W}^r_d$, which was proved  and used effectively in \cite{I}. 
A base-point-free linear series $g^r_d$ on $C$  is called compounded of an involution (compounded for short) if the morphism induced by the linear series gives rise to a non-trivial covering map $C\rightarrow C'$ of degree $k\ge 2$. 

\begin{prop}[\rm{\cite[Proposition 2.1]{I}}]\label{wrdbd}
Let $d,g$ and $r\ge 2$ be positive integers such that  $d\le g+r-2$ and let $\mathcal{W}$ be an irreducible component of $\mathcal{W}^{r}_{d}$. For a general element $(p,L)\in \mathcal{W}$, let $b$ be the degree of the base locus of the line bundle $L=|D|$ on $C=\xi^{-1}(p)$. Assume further that for a general $(p,L)\in \mathcal{W}$ the curve $C=\xi^{-1}(p)$ is not hyperelliptic. If the moving part of $L=|D|$ is
	\begin{itemize}
	\item[\rm{(a)}] very ample and $r\ge3$, then 
	$\dim \mathcal{W}\le 3d+g+1-5r-2b$;
	\item[\rm{(b)}] birationally very ample, then 
	$\dim \mathcal{W}\le 3d+g-1-4r-2b$;
	\item[\rm{(c)}] compounded, then 
	$\dim \mathcal{W}\le 2g-1+d-2r$.
	\end{itemize}
\end{prop}

The authors are grateful to the referee for several valuable comments and useful  suggestions toward an overall improvement of the clarity of this paper.

\section{\quad Irreducibility of $\mathcal{H}^\mathcal{L}_{d,g,r}$ for $g+r-3\le d \le g+r$}

Needless to say, in case $d\ge g+r+1$ there is no  complete linear system of degree $d$ and dimension $r$ by the Riemann-Roch formula. Therefore in this range we have $\mathcal{H}^\mathcal{L}_{d,g,r}=\emptyset$ and  the Modified Assertion of Severi makes sense only if $g-d+r\ge 0$. The first object  to be considered is $\mathcal{H}^\mathcal{L}_{g+r,g,r}$ whose irreducibility is quite obvious.

\begin{thm} $\mathcal{H}^\mathcal{L}_{g+r,g,r}$ is non-empty and irreducible.
\end{thm}
\begin{proof} Recall that inside the Brill-Noether range, i.e. in the range $\rho (d,g,r)\ge 0$,  there is only one component of $\mathcal{H}_{d,g,r}$ dominating $\mathcal{M}_g$ which one calls the principal component; cf. \cite[2.1 page 70]{H1} or Remark \ref{principal}. Furthermore, the principal component $\mathcal{H}_0$ has the expected dimension, which is a component of $\mathcal{H}^\mathcal{L}_{d,g,r}$. We further note that in the Brill-Noether range $\rho (d,g,r)\ge 0$, any  possible component $\mathcal{H}$ of $\mathcal{H}_{d,g,r}$ other than $\mathcal{H}_0$ consists of curves in $\mathbb{P}^r$ whose  hyperplane series is special, which seems to be rather well-known. However the authors could not find an adequate source of a proof in the literature and therefore a simple argument is provided as follows.
 Let $\mathcal{H}\neq\mathcal{H}_0$ be an irreducible component of $\mathcal{H}_{d,g,r}$. By Proposition \ref{facts}(1), $\mathcal{H}$ is a $\PP GL(r+1)$-bundle over an open subset of a component $\mathcal{G} ( \neq\mathcal{G}_0)$ of $\widetilde{\mathcal{G}}$ not dominating $\mathcal{M}_g$. Let $C$ be a smooth irreducible non-degenerate curve of degree $d$ and genus $g$ in $\PP^r$ corresponding to a general point $c \in \mathcal{H}$. We claim that $\mathcal{O}_C(1)$ is special.  If $\mathcal{O}_C(1)$ is non-special then, by Riemann-Roch, $d \ge g + r$. Considering the (non-dominating) natural rational projection map 
$\mathcal{G} \overset{\eta}\dashrightarrow \mathcal{M}_g$, 
\begin{eqnarray*}
 \dim\mathcal{G}&\le&\dim\eta(\mathcal{G})+\dim G^r_d(C)<\dim\mathcal{M}_g+\dim G^r_d(C)\\
 &\le&\dim\mathcal{M}_g+\dim J(C)+\dim\mathbb{G}(r,d-g)\\&=&3g-3+g+(r+1)(d-g-r)=3g-3+\rho (d,g,r),
 \end{eqnarray*} which is contradictory to the fact $\dim\mathcal{G}\ge \lambda (d,g,r) =3g-3+\rho (d,g,r)$ (Proposition \ref{facts}(1)), completing the proof of the claim. 
 
 In the case $d=g+r$, a curve which is embedded in $\mathbb{P}^r$ by a special linear series is not linearly normal. Therefore the only component of $\mathcal{H}^\mathcal{L}_{g+r,g,r}$ is the principal component. 	The non-emptiness of $\mathcal{H}^\mathcal{L}_{g+r,g,r}$ follows from \cite[Ch. IV, 3.3.1, 3.3.3 and Proposition 6.1]{Hartshorne}.
\end{proof}

\begin{thm}\label{g+r-1}
$\mathcal{H}^\mathcal{L}_{g+r-1,g,r}$ is  non-empty and irreducible for $g\ge r+1$ and is empty for $g\le r$.
\end{thm}

\begin{proof} 
Let ${\mathcal{G}}\subset \widetilde{\mathcal{G}}_\mathcal{L}\subset \widetilde{\mathcal{G}}\subset\mathcal{G}^{r}_{g+r-1}$ be an irreducible component whose general element $(p, \mathcal{D})$ is a very ample and complete linear series $\mathcal{D}$ on the curve $C=\xi^{-1}(p)$. 
For a general $(p, \mathcal{D})\in\mathcal{G}$, the residual series $|K_C-\mathcal{D}|=g^{0}_{g-r-1}$ is an effective divisor of degree $g-r-1$. Conversely, for a general effective divisor $p_1+\cdots +p_{g-r-1}\in C_{g-r-1}$ on a non-hyperelliptic curve $C$,  a  linear series which is of the form $|K_C-p_1-\cdots -p_{g-r-1}|$ is a complete and very ample $g^r_{g+r-1}$. Therefore one may easily deduce  that $\widetilde{\mathcal{G}}_\mathcal{L}$ is birational to the irreducible locus $\mathcal{G}^{0}_{g-r-1}$ and so is  $\mathcal{H}^\mathcal{L}_{g+r-1,g,r}$ which is a $\mathbb{P}\textrm{GL}(r+1)$-bundle over an open subset of $\widetilde{\mathcal{G}}_\mathcal{L}$. The non-emptiness is rather clear; a general member of $\mathcal{H}^\mathcal{L}_{g+r-1,g,r}$ corresponds to a curve of degree $g+r-1$ in $\mathbb{P}^r$ obtained by successive $g-r-1$ general projections from the canonical curve.  By the Castelnuovo genus bound \cite[page 116]{ACGH}, for $d=g+r$ and $g\le r$ one has $\pi (d,r)=g-1$. Therefore it follows that $\mathcal{H}^\mathcal{L}_{g+r-1,g,r}=\emptyset$ for $g\le r$.
\end{proof}

\begin{thm}\label{g+r-2} Every non-empty $\mathcal{H}^\mathcal{L}_{g+r-2,g,r}$ is irreducible.
\end{thm}
\begin{proof} We first estimate the dimension of a component ${\mathcal{G}}\subset \widetilde{\mathcal{G}}_\mathcal{L}\subset \widetilde{\mathcal{G}}\subset\mathcal{G}^{r}_{g+r-2}$. By Proposition \ref{facts} (1), we have 
\[
\lambda (g+r-2,g,r)=3g-3+\rho(g+r-2,g,r)=4g-2r-5\le \dim {\mathcal{G}}.
\]
Note that  $r=h^0(C, |\mathcal{D}|)-1$ for a general $(p,\mathcal{D})\in{\mathcal{G}}$.
Let $\mathcal{W}\subset \mathcal{W}^{r}_{g+r-2}$ be the component containing the image of the natural rational map 
${\mathcal{G}}\overset{\iota}{\dashrightarrow} \mathcal{W}^{r}_{g+r-2}$ with $\iota (\mathcal{D})=|\mathcal{D}|$.
Since $\dim{\mathcal{G}}\le\dim \mathcal{W}$, it follows  by Proposition  \ref{wrdbd} (1) that 
\[
\lambda (g+r-2,g,r)=4g-2r-5\le \dim \mathcal{G}\le 3(g+r-2)+g+1-5r=4g-2r-5,
\]
and hence
\begin{equation}\label{pencil}
\dim\mathcal{G}=\dim\mathcal{W}=\lambda (g+r-2,g,r)=4g-2r-5.
\end{equation}
We let $\mathcal{W}^\vee\subset \mathcal{W}^{1}_{g-r}$ be the locus consisting  of the residual  series (with respect to the canonical series on the corresponding curve) 
of those elements in $\mathcal{W}$, i.e. $\mathcal{W}^\vee =\{(p, \omega_C\otimes L^{-1}): (p, L)\in\mathcal{W}\}.$
We also let $\widetilde{\mathcal{W}}^\vee\subset\mathcal{W}^1_{g-r}$ be the union of those  components $\mathcal{W}^\vee$ corresponding to each $\mathcal{W}$ arising from some ${\mathcal{G}}\subset \widetilde{\mathcal{G}}_\mathcal{L}$.
By our previous dimension count  (\ref{pencil}), 
\begin{equation}\label{scriptpencil}\dim \mathcal{W}^{\vee}=\dim\mathcal{W}=\dim\mathcal{G}=4g-2r-5=\lambda (g-r,g,1)=\dim \mathcal{G}^1_{g-r}.\end{equation}
Since a general element of any component $\mathcal{W}^{\vee}\subset\widetilde{\mathcal{W}}^\vee\subset\mathcal{W}^1_{g-r}$ is 
a complete pencil by our setting, there is a natural rational map $\widetilde{\mathcal{W}}^\vee\overset{\kappa}{\dashrightarrow}\mathcal{G}^1_{g-r}$ with $\kappa(|\mathcal{E}|)=\mathcal{E}$   which is clearly injective on an open subset  $\widetilde{\mathcal{W}}^{\vee o}$ of $\widetilde{\mathcal{W}}^\vee$ consisting of those which are
complete pencils. Therefore the  rational map $\kappa$ is dominant by (\ref{scriptpencil}). 
We also note that there is another natural rational map 
$\mathcal{G}^1_{g-r} \overset{\iota}{\dashrightarrow}\widetilde{\mathcal{W}}{^\vee}$ with $\iota (\mathcal{E})=|\mathcal{E|}$, which is an inverse to $\kappa$ (wherever it is defined).
Therefore it follows that $\widetilde{\mathcal{W}}^{\vee}$ is birationally equivalent to  the irreducible locus $\mathcal{G}^1_{g-r}$ (cf. Proposition \ref{facts}(3)), hence $\widetilde{\mathcal{W}}^{\vee}$ is irreducible and so is $\widetilde{\mathcal{G}}_\mathcal{L}$. Since $\mathcal{H}_{g+r-2,g,r}^\mathcal{L}$ is a $\mathbb{P}\textrm{GL}(r+1)$-bundle over an open subset of $\widetilde{\mathcal{G}}_\mathcal{L}$,  $\mathcal{H}_{g+r-2,g,r}^\mathcal{L}$ is irreducible.
\end{proof}

\begin{rmk} 
\begin{enumerate}
\item[\rm{(i)}] $\mathcal{H}^\mathcal{L}_{g+r-2,g,r}=\emptyset$ for $g\le r+2$, $r\ge 3$ by the Castelnuovo genus bound. 
\vskip 4pt
\noindent
\item[\rm{(ii)}] For $g=r+3$, a curve of degree $d=g+r-2=2r+1$ and genus $g$ in $\mathbb{P}^r$ is an extremal curve, which is clearly linearly normal. Hence $\mathcal{H}^\mathcal{L}_{2r+1,r+3,r}=\mathcal{H}_{2r+1,r+3,r}\neq\emptyset$ and is irreducible by \cite[Theorem 1.4(i)]{CC} or by our Theorem \ref{g+r-2}.

\vskip 4pt
\noindent
\item[\rm{(iii)}] In the Brill-Noether range $\rho (g+r-2,g,r)=g-2(r+1)\ge 0$, $\mathcal{H}^\mathcal{L}_{g+r-2,g,r}\neq\emptyset$; cf. Remark \ref{principal} (2).

\vskip 4pt
\noindent
\item[\rm{(iv)}] Even outside the Brill-Noether range, e.g. $g=2r$ or  
$g=2r+1$, we have $\mathcal{H}^\mathcal{L}_{g+r-2,g,r}\neq\emptyset$ which is irreducible of the expected dimension. By  \cite[Corollary 1.3, Remark 1.5]{BE}, for every triple $(d,g,r)$  in the range
$$d-r<g\le d-r+\lfloor (d-r-2)/(r-2)\rfloor, d\ge r\ge 3$$ there is a certain reducible curve $C_0\subset \mathbb{P}^r$ of degree $d$ and (arithmetic) genus $g$ such that $h^1(N_{C_{0}})=0$ and $C_0$ is smoothable.  We note that  $g=2r$ ($2r+1$ resp.)  with $d=g+r-2=3r-2$ ($3r-1$ resp.) is in the above range. Therefore a component $\mathcal{H}$ of $\mathcal{H}_{d,g,r}$ containing the reducible curve $C_0$ exists which is of the expected dimension. Furthermore it can be shown that $h^0(C_0, \mathcal{O}_{C_0}(1))=r+1$ and hence by semi-continuity a general element of $\mathcal{H}$ is linearly normal so that $\mathcal{H}^{\mathcal{L}}_{g+r-2,g,r}\neq\emptyset$ which is irreducible by our Theorem \ref{g+r-2}.

\vskip 4pt
\noindent
\item[\rm{(v)}] For low genus $g$ with respect to $r$, we have only scattered examples. For example, take
$r=8$, $g=r+4=12$, $d=g+r-2 =18$. Let $S\subset \PP^8$ be the Del Pezzo surface isomorphic to the blowing up of $\PP^2$ at one point $o$ anti-canonically embedded by 
$|\mathcal{I} _{o,\PP^2}(3)|$. The curve $X$ which is the image of a degree $7$ plane curve with an ordinary triple point at  the point $o$ and no other singularity 
has degree $d=21-3=18$ and genus $g=(7-1)(7-2)/2 -3 =12$, hence $X\in \mathcal{H}^\mathcal{L}_{g+r-2,g,r}$.
\end{enumerate}
\end{rmk}
\begin{thm}\label{g+r-3} 
\begin{enumerate}  \item[\rm{(i)}]  Every non-empty $\mathcal{H}^\mathcal{L}_{g+r-3,g,r}$ is irreducible for $g\ge 2r+3$.

\item[\rm{(ii)}]  $\mathcal{H}^\mathcal{L}_{g+r-3,g,r}$ is non-empty and irreducible  for $g\ge 3r+3$.
\end{enumerate}
\end{thm}
\begin{proof}  (i) We first show the irreducibility of  $\mathcal{H}^\mathcal{L}_{g+r-3,g,r}$ for $g\ge 2r+3$. As in Theorem \ref{g+r-2}, we need to estimate the dimension of a component 
${\mathcal{G}}\subset \widetilde{\mathcal{G}}_\mathcal{L}\subset \widetilde{\mathcal{G}}\subset\mathcal{G}^{r}_{g+r-3}$.
By Proposition \ref{facts} (1), we have 
\begin{equation}\label{net}
\lambda (g+r-3,g,r)=3g-3+\rho(g+r-3,g,r)=4g-3r-6\le \dim {\mathcal{G}}.
\end{equation}
Note that  $r=h^0(C, |\mathcal{D}|)-1$ for a general $(p,\mathcal{D})\in{\mathcal{G}}$.
Let $\mathcal{W}\subset \mathcal{W}^{r}_{g+r-3}$ be a component containing the image of the natural rational map 
${\mathcal{G}}\overset{\iota}{\dashrightarrow} \mathcal{W}^{r}_{g+r-3}$ with $\iota (\mathcal{D})=|\mathcal{D}|$.
We also let $\mathcal{W}^\vee\subset \mathcal{W}^{2}_{g-r+1}$ be the locus consisting  of the residual  series  of elements in $\mathcal{W}$, i.e. $\mathcal{W}^\vee =\{(p, \omega_C\otimes L^{-1}): (p, L)\in\mathcal{W}\}.$

\vskip 4pt
\noindent

\begin{enumerate}
\item[\rm{(a)}]
If a general element of $\mathcal{W}^{\vee}$ is compounded, then by Proposition \ref{wrdbd}(c),
\begin{eqnarray*}
4g-3r-6&\le&\dim\mathcal{G}~\le~\dim \mathcal{W}=\dim \mathcal{W}^{\vee}\\
&\le& 2g-1+(g-r+1)-2\cdot 2\\
&=&3g-r-4
\end{eqnarray*}
implying  $g\le 2r+2$ contrary to our assumption $g\ge 2r+3$. 
Therefore we conclude that a general element of $\mathcal{W}^{\vee}$ is either very ample or birationally very ample. 
\vskip 4pt
\noindent
\item[\rm{(b)}]
Suppose that the moving part of a  general element of $\mathcal{W}^{\vee}\subset\mathcal{W}^{2}_{g-r+1}$ is very ample and let $b$ be the degree of the base locus $B$ of a general element of $\mathcal{W}^{\vee}$. Putting $e:=g-r+1-b$ and $g=\frac{(e-1)(e-2)}{2}$ a simple numerical calculation yields

\begin{align*}\hskip 18pt\dim\mathcal{G}&\le\dim \mathcal{W}=\dim \mathcal{W}^{\vee}\le\dim\mathbb{P}(H^0(\mathbb{P}^2, \mathcal{O}(e)))-\dim\mathbb{P}\rm{GL}(3)+b\\&=\frac{(e+1)(e+2)}{2}-9+b
<\lambda(g+r-3,g,r)=4g-3r-6
\end{align*}
contrary to the inequality (\ref{net}); note that the above inequality is equivalent to  $2r < e^2-5e+4$, which follows from our genus bound  $2r \le g-3 = \frac{(e-1)(e-2)}{2}-3$ and 
$r\ge 3$.
Such a contradiction excludes the case (b). 

\vskip 4pt
\noindent
\item[\rm{(c)}]
Therefore the moving part of a  general element of $\mathcal{W}^{\vee}\subset\mathcal{W}^{2}_{g-r+1}$ is birationally very ample and we let $b$ be the degree of the base locus $B$ of a general element of $\mathcal{W}^{\vee}$. 
 By Proposition \ref{wrdbd}(b), we have
\begin{eqnarray*}
4g-3r-6&\le&\dim\mathcal{G}\le\dim \mathcal{W}= \dim \mathcal{W}^{\vee}\\
&\le& 3(g-r+1)+g-1-4\cdot 2-2b\\
&=&4g-3r-6-2b, 
\end{eqnarray*}
implying $b=0$. Thus it finally follows that  
\begin{equation}\label{equall}
	\dim\mathcal{G}=\dim\mathcal{W}=\dim\mathcal{W}^{\vee}=4g-3r-6=\lambda (g+r-3,g,r).
\end{equation}
\end{enumerate}
\vskip 5pt
\noindent
Retaining the same notation as in the proof of Theorem \ref{g+r-2}, let $\widetilde{\mathcal{G}}_\mathcal{L}$ be the union of irreducible components $\mathcal{G}$ of $\mathcal{G}^{r}_{g+r-3}$ whose general element corresponds to a pair $(p,\mathcal{D})$ such that $\mathcal{D}$ is very ample and complete  linear series on $C:=\xi^{-1}(p)$. Let $\widetilde{\mathcal{W}}^\vee$ be the union of the  components $\mathcal{W}^\vee$ of $\mathcal{W}^2_{g-r+1}$, where $\mathcal{W}^\vee$ consists of the residual series of elements in $\mathcal{G}\subset\tilde{\mathcal{G}}_\mathcal{L}$.
We also let $\mathcal{G}'$ be the union of irreducible components of $\mathcal{G}^{2}_{g-r+1}$ whose general element is a birationally very ample and  base-point-free linear series.  We recall that, by Lemma \ref{Gisirred} and Proposition \ref{facts}(2), $\mathcal{G'}$ is irreducible and $\dim \mathcal{G'}=3(g-r+1)+g-9=4g-3r-6$.  
By  the equality (\ref{equall}), \begin{equation}\label{dominantt}\dim \mathcal{W}^{\vee}=\dim\mathcal{G}=4g-3r-6=\dim \mathcal{G}'.\end{equation}
Since a general element of any component $\mathcal{W}^{\vee}\subset\widetilde{\mathcal{W}}^\vee\subset\mathcal{W}^2_{g-r+1}$ is a base-point-free,  birationally very ample and complete net, there is a natural rational map $\widetilde{\mathcal{W}}^\vee\overset{\kappa}{\dashrightarrow}\mathcal{G}'$ with $\kappa(|\mathcal{E}|)=\mathcal{E}$   which is clearly injective on an open subset  $\widetilde{\mathcal{W}}^{\vee o}$ of $\widetilde{\mathcal{W}}^\vee$ consisting of those which are base-point-free,  birationally very ample and complete nets. Therefore the  rational map $\kappa$ is dominant by (\ref{dominantt}).
We also note that there is a natural rational map 
$\mathcal{G}' \overset{\iota}{\dashrightarrow}\widetilde{\mathcal{W}}{^\vee}$ with $\iota (\mathcal{E})=|\mathcal{E|}$, which is an inverse to $\kappa$ (wherever it is defined).
Therefore it follows that $\widetilde{\mathcal{W}}^{\vee}$ is birationally equivalent to  the irreducible locus $\mathcal{G}'$, hence $\widetilde{\mathcal{W}}^{\vee}$ is irreducible and so is $\widetilde{\mathcal{G}}_\mathcal{L}$. Since $\mathcal{H}_{g+r-3,g,r}^\mathcal{L}$ is a $\mathbb{P}\textrm{GL}(r+1)$-bundle over an open subset of $\widetilde{\mathcal{G}}_\mathcal{L}$,  $\mathcal{H}_{g+r-3,g,r}^\mathcal{L}$ is irreducible.

(ii) Note that $\rho (g+r-3,g,r)\ge 0$ for $g\ge 3r+3$ and the non-emptiness of $\mathcal{H}^\mathcal{L}_{g+r-3,g,r}$ for $g\ge 3r+3$ follows from \cite[Theorem 1..8]{EH3}; cf. \cite[p. 844]{ACGH2}. 
\end{proof}

\begin{rmk}
\begin{enumerate}
\item[\rm{(i)}] For $g\le r+4$, $d=g+r-3$, $r\ge 4$, $\mathcal{H}_{d,g,r}=\mathcal{H}_{g+r-3,g,r}=\emptyset$ ($d\le 2r+1$) by the Castelnuovo genus bound. 
\vskip 4pt
\noindent
\item[\rm{(ii)}]
For the next low genus $g$ with respect to the dimension of the ambient projective space $\mathbb{P}^r$ - say $g=r+5$ - the Hilbert scheme of linearly normal curves $\mathcal{H}^\mathcal{L}_{g+r-3,g,r}=\mathcal{H}^\mathcal{L}_{2r+2,r+5,r}$ is reducible only when $r=4,5$. Indeed,  a curve of genus $g=r+5$ and degree $d=g+r-3=2r+2$ in $\mathbb{P}^r$ ($r\ge 4$) is an extremal curve. 
The case $r=4, g=9,  d= 10$ was treated in \cite[Theorem 2.2]{KK3}, in which case the corresponding Hilbert scheme of linearly normal curves has two components of different dimensions. For $r=5$, $g=10$ and $d=12$ the corresponding Hilbert scheme $\mathcal{H}^\mathcal{L}_{12,10,5}$ of linearly normal curves is reducible with two components; cf.  \cite[Theorem 1.4]{CC}.  For $r\ge 6$, $\mathcal{H}^\mathcal{L}_{2r+2,r+5,r}$ is irreducible also by \cite[Theorem 1.4(i)]{CC}.

\end{enumerate}
\end{rmk}

\begin{exam}
The restriction we imposed on the genus $g\ge 2r+3$ in Theorem \ref{g+r-3} is rather optimal and cannot be eased in general as the following example indicates.  Recall that the restriction was used in two places in the proof of Theorem \ref{g+r-3}; (a) when a general element of $\mathcal{W}^\vee$ is compounded and (b) when the moving part of a general element of $\mathcal{W}^\vee$ is very ample. For the sake of simplicity, let's take $r=5$,  $g=2r+2=12$ and $d=g+r-3=14$. 

\vskip 4pt
\noindent
(a) If general element $\mathcal{E}$ of $\mathcal{W}^\vee\subset \mathcal{W}^2_8$ is compounded, one of the following may occur:

\noindent
\begin{enumerate}
\item[\rm{(1)}] $C$ is trigonal and $\mathcal{E}=2g^1_3+\Delta$, where $\Delta$ is the base locus of $\mathcal{E}$.

\noindent
\item[\rm{(2)}]
$C$ is bi-elliptic with a bi-elliptic covering $C\stackrel{\phi} \rightarrow E$ and $\mathcal{E}=\phi^*(g^2_3)+\Delta$. 

\noindent
\item[\rm{(3)}] $C$ is a double covering of a curve of genus $2$ with a double covering $C\stackrel{\zeta} \rightarrow E$ and $\mathcal{E}=\zeta^*({g^2_4})$.

\noindent
\item[\rm{(4)}] $C$ is a double covering of a smooth plane quartic $F$ with the double covering $C\stackrel{\eta}\rightarrow F$ and $\mathcal{E}=\eta^*(|K_F|)=\eta^*(g^2_4)$.

\noindent
\item[\rm{(5)}] $C$ a $4$-gonal curve with $\mathcal{E}=2g^1_4$.  
\end{enumerate}
\vskip 4pt
\noindent
In the first three cases (1), (2),  (3), one may easily argue that the residual series $\mathcal{D}=|K_C-\mathcal{E}|$ is not very ample. 
For example if  $C$ is trigonal, we take the residual series $\mathcal{D}$ of $\mathcal{E}=|2g^1_3+q+t|=g^2_8$; note that 
\begin{eqnarray*} \dim|\mathcal{D}-r-s|&=&|K_C-2g^1_3-q-t-r-s|=
\dim|K_C-3g^1_3-t|\\&\ge& \dim|K_C-3g^1_3|-1>\dim\mathcal{D}-2
\end{eqnarray*}
where $q+r+s\in g^1_3$.

\vskip 4pt
\noindent
In the case (2), we note that
$$\dim|\mathcal{E}+\phi^*(p)|=\dim|\phi^*(g^2_3)+\Delta +\phi^*(p)|=\dim|\phi^*(g^3_4)+\Delta|\ge 3$$
hence $\mathcal{D}=|K_C-\mathcal{E}|$ is not very ample. The case (3) is similar.

\vskip 4pt
\noindent
In the case (4), we 
have a natural rational map $\mathcal{W}\dashrightarrow \mathcal{W}^\vee\dashrightarrow \mathcal{X}_{2,3}$, where $\mathcal{X}_{n,\gamma}$ denotes the locus in $\mathcal{M}_g$ corresponding to curves
which are n-fold coverings of smooth curves of genus $\gamma$. By applying  Accola-Griffiths-Harris theorem \cite[page 73]{H1} to a general fiber of the above rational map together with Riemann's moduli count \cite[Satz 1]{Lange} $$\dim\mathcal{X}_{n,\gamma}\le 2g+(2n-3)(1-\gamma)-2,$$ one has
\begin{eqnarray*}
\dim\mathcal{W}&\le&\dim\mathcal{X}_{2,3}+2d-g-3r+1 \\
&\le& (2g-4)+2d-g-3r+1=g+2d-3r-3, 
\end{eqnarray*}
contrary to $\dim\mathcal{W}\ge\lambda (d,g,r)$. 

\vskip 4pt
\noindent
In the case (5) when
$C$ is $4$-gonal with $\mathcal{E}=2g^1_4$, one has a generically injective rational map  $\mathcal{W}\dashrightarrow \mathcal{W}^\vee\dashrightarrow \mathcal{M}^1_{g,4}$ and hence
$\dim\mathcal{W}\le\dim\mathcal{M}^1_{g,4}=2g+3$. Therefore we may conclude that 
$\dim\mathcal{G}=\dim\mathcal{W}=\lambda (d,g,r)=2g+3$  in which case one may use the irreducibility of $\mathcal{M}^1_{g,4}$ instead of the irreducibility of the Severi variety.

\vskip 4pt
\noindent (b) It is clear that the case when the moving part of a general element of $\mathcal{W}^\vee\subset\mathcal{W}^2_8$ is very ample does not occur
 just because there is no smooth plane curve of genus $12=2r+2$. 
 
 \vskip 4pt
 \noindent
 In all, we have the following two possibilites: 
 \begin{enumerate}
 \item[\rm{(i)}]
 For general $(p,\mathcal{D})\in\mathcal{G}$, $|K_C-\mathcal{D}|=\mathcal{E}=2g^1_4$.
 
 \item[\rm{(ii)}]
 For a general $(p,\mathcal{D})\in\mathcal{G}$, $|K_C-\mathcal{D}|=\mathcal{E}=g^2_8$ is base-point-free and birationally very ample.
 \end{enumerate}
 \vskip 4pt
 \noindent
 By the irreducibility of $\mathcal{M}^1_{g,4}$ and  the Severi variety $\Sigma_{14,12}$ we see that there are exactly two components $\mathcal{H}_1$ and $\mathcal{H}_2$ of $\mathcal{H}^\mathcal{L}_{14,12,5}$ corresponding to the possibilities (i) and (ii). It is also clear that 
 $\mathcal{H}_1\neq\mathcal{H}_2$ by semi-continuity and $\dim\mathcal{H}_1=\dim\mathcal{H}_2=\lambda (14,12,5)+\dim\mathbb{P}GL(6)$. 
\end{exam}

\bibliographystyle{spmpsci} 

\begin{thebibliography}{111}
\bibitem{AC2}
{E. Arbarello and M. Cornalba},
\textit{A few remarks about the variety of irreducible plane curves of given degree and genus.} Ann. Sci. \'Ec. Norm. Sup\'er. (4) \textbf{16} (1983), 467--483.
\bibitem{ACGH}
{E. Arbarello, M. Cornalba, P. Griffiths and J. Harris},
\textit{Geometry of Algebraic Curves Vol.I.}
Springer-Verlag, Berlin/Heidelberg/New York/Tokyo, 1985.
\bibitem{ACGH2}
{E. Arbarello, M. Cornalba and  P. Griffiths},
\textit{Geometry of Algebraic Curves Vol.II.}
Springer, Heidelberg, 2011.
\bibitem{BE}
{E. Ballico and Ph. Ellia},
\textit{On the existence of curves with maximal rank in $\mathbb{P}^n$.} 
J. reine angew. Math. \textbf{397} (1989), 1--22.
\bibitem{CC}
{C. Ciliberto},
\textit{On the Hilbert Scheme of Curves of Maximal Genus in a Projective Space.} 
Mathematische Zeitschrift \textbf{194} (1987), 451--463.
\bibitem{CS}
{C. Ciliberto and E. Sernesi},
{\it Families of varieties and the Hilbert scheme.}
Lectures on Riemann surfaces (Trieste, 1987), 428--499, World Sci. Publ., Teaneck, NJ, 1989
\bibitem{E1}
{L. Ein},
\textit{Hilbert scheme of smooth space curves}.
Ann. Sci. \'Ec. Norm. Sup\'er. (4), \textbf{19} (1986), no. 4, 469--478.
\bibitem{E2}
{L. Ein},
{\it The irreducibility of the Hilbert scheme of complex space curves.}
Algebraic geometry, Bowdoin, 1985 (Brunswick, Maine, 1985), Proc. Sympos. Pure Math., 46, Part 1, Providence, RI: Amer. Math. Soc., 83--87.
\bibitem{EH}
{D. Eisenbud and J. Harris}, 
\textit{Irreducibility and monodromy of some families of linear series.} 
Ann. Sci. \'Ec. Norm. Sup\'er. (4), \textbf{20} (1987), no. 1, 65--87.
\bibitem{EH2}
{D. Eisenbud and J. Harris}
\textit{3264 and All That,
A Second Course in Algebraic Geometry.}, Cambridge University Press, Cambridge, 2016.
\bibitem{EH3}
{D. Eisenbud and J. Harris},
\textit{Divisors on general curves and cuspidal rational curves}.
Invent. Math., \textbf{74} (1983), 371--418.
\bibitem{H1} 
{J. Harris},
\textit{Curves in Projective space.}
in ``Sem. Math. Sup.,", Press Univ. Montr\'eal, Montr\'eal, 1982.
\bibitem{H2}
{J. Harris},
\textit{On the Severi problem.}
Invent. Math \textbf{84}, (1986), 445--461.
\bibitem{Hartshorne}
{R. Hartshorne},
\textit{Algebraic Geometry.}
Springer-Verlag, Berlin/Heidelberg/New York, 1977.
\bibitem{I}
{H. Iliev},
\textit{On the irreducibility of the Hilbert scheme of space curves.} Proc. Amer. Math. Soc., \textbf{134} (2006), no. 10, 2823--2832.
\bibitem{Keem}
{C. Keem}, 
\textit{Reducible Hilbert scheme of smooth curves with positive Brill-Noether number.}
Proc. Amer. Math. Soc., \textbf{122} (1994), no. 2, 349--354.
\bibitem{KK}
{C. Keem and S. Kim},
\textit{Irreducibility of a subscheme of the Hilbert scheme of complex space curves.} 
J. Algebra, \textbf{145} (1992), no. 1, 240--248.
\bibitem{KKy1}
{C. Keem and Y.-H. Kim},
\textit{Irreducibility of the Hilbert Scheme of smooth curves in $\PP^3$ of degree $g$ and genus $g$.} 
Arch. Math., \textbf{108} (2017), no. 6, 593--600.
\bibitem{KKy2}
{C. Keem and Y.-H. Kim},
\textit{Irreducibility of the Hilbert Scheme of smooth curves in $\PP^4$ of degree $g+2$ and genus $g$.} 
Arch. Math., \textbf{109} (2017), no. 6, 521--527.
\bibitem{KK3}
{C. Keem and Y.-H. Kim},
\textit{On the Hilbert scheme of linearly normal curves in $\mathbb{P}^4$ of degree $d = g+1$ and genus $g$.}
Arch. Math. (2019), \url{https://doi.org/10.1007/s00013-019-01337-2}.
\bibitem{KKL}
{C. Keem, Y.-H. Kim and A.F. Lopez},
\textit{Irreducibility and components rigid in moduli of the Hilbert Scheme of smooth curves.} 
Math. Z. (2018), \url{https://doi.org/10.1007/s00209-018-2130-1}.
\bibitem{Lange}
{H. Lange}, 
\textit{Kurven mit rationaler Abbildung.}
J. reine angew. Math.,
\textbf{295} (1977), 80-115.
\bibitem{MS}
{E. Mezzetti and G. Sacchiero},
\textit{Gonality and Hilbert schemes of smooth curves.} 
Algebraic curves and projective geometry (Trento, 1988), 183--194, 
Lecture Notes in Math., 1389, Springer, Berlin, 1989. 
\bibitem{Sev1}
{F.  Severi}, 
\textit{Sulla classificazione delle curve algebriche e sul teorema di esistenza di Riemann.}
Rend. R. Acc. Naz. Lincei, \textbf{241} (1915), 
877--888.
\bibitem{Sev}
{F.  Severi}, 
\textit{Vorlesungen \"uber algebraische Geometrie.}
Teubner, Leipzig, 1921.

\end{thebibliography}

\end{document}